\documentclass{amsart}
\usepackage{a4wide,amsmath,amssymb}
\usepackage[toc,page]{appendix}
\usepackage{hyperref}
\usepackage{url}
\usepackage{mathtools}
\usepackage{xcolor}

\newtheorem{thm}{Theorem}[section]

\newtheorem{cor}[thm]{Corollary}
\theoremstyle{remark}
\newtheorem{rem}[thm]{Remark}

\theoremstyle{definition}

\newtheoremstyle{Claim}{}{}{\itshape}{}{\itshape\bfseries}{:}{ }{#1}
\theoremstyle{Claim}

\newcommand{\T}{{\mathbb{T}}}
\newcommand{\Z}{{\mathbb{Z}}}

\newcommand{\R}{\mathbb{R}}

\newcommand{\eps}{\varepsilon}

\makeatletter
\@namedef{subjclassname@2020}{\textup{2020} Mathematics Subject Classification}
\makeatother

\theoremstyle{plain}

\def\sideremark#1{\ifvmode\leavevmode\fi\vadjust{
\vbox to0pt{\hbox to 0pt{\hskip\hsize\hskip1em
\vbox{\hsize3cm\tiny\raggedright\pretolerance10000
\noindent #1\hfill}\hss}\vbox to8pt{\vfil}\vss}}}

\begin{document}

\title[]{Transport-diffusion equations with irregular data and applications to stability estimates for second-order Hamilton-Jacobi PDEs}

\author{Gianmarco Giovannardi}
\address{Dipartimento di Matematica e Informatica ``Ulisse Dini'', Universit\`a degli Studi di Firenze, 
viale G. Morgagni 67/a, 50134 Firenze (Italy)}
\curraddr{}
\email{gianmarco.giovannardi@unifi.it}
\author{Alessandro Goffi}
\address{Dipartimento di Matematica e Informatica ``Ulisse Dini'', Universit\`a degli Studi di Firenze, 
viale G. Morgagni 67/a, 50134 Firenze (Italy)}
\curraddr{}
\email{alessandro.goffi@unifi.it}

\subjclass[]{}
 \thanks{GG is supported by MIUR-PRIN 2022 Project \emph{Geometric-Analytic Methods
for PDEs and Applications}. The authors are members of the Gruppo Nazionale per l'Analisi Matematica, la Probabilit\`a e le loro Applicazioni (GNAMPA) of the Istituto Nazionale di Alta Matematica (INdAM) and they are partially supported by INdAM-GNAMPA 2026 Projects   \emph{Parabolic equations with drift: well posedness, gradient estimates and boundary effects} and \emph{Processi di diffusione non-lineari: regolarit\`a e classificazione delle soluzioni}
 }
\date{\today}



\begin{abstract}
This paper studies quantitative uniqueness properties in $L^p$ spaces for Fokker-Planck and transport-diffusion equations under two new assumptions on their velocity field $b=b(x,t)$. We first prove $L^p$-stability estimates for advection-diffusion PDEs when $\mathrm{div}(b)\in L^r_t(L^q_x)$ with $r\in[2,\infty]$ and $q\in[n/2,\infty)$ satisfying the compatibility condition $n/(2q)+1/r\leq 1$. We then prove a stability result in $L^\infty$ for solutions of viscous transport equations when $\mathrm{div}(b(t))$ fails to be integrable in time. We apply these properties to obtain new continuous dependence estimates for viscous Hamilton-Jacobi equations via integral methods. One of the main novelties in this latter setting is that the constants of the estimates are all explicit with respect to the data of the problem. These imply new uniqueness properties for diffusive Hamilton-Jacobi equations without relying on the theory of viscosity solutions.
\end{abstract}

\date{\today}
\maketitle

\section{Introduction}
In this paper we address a priori stability estimates and uniqueness for the viscous conservative equation with periodic boundary conditions
\begin{equation}\label{fp}
\partial_t\rho-\eps\Delta \rho+\mathrm{div}(b(x,t)\rho)=0
\end{equation}
and discuss their byproducts to its adjoint version, namely the drift-diffusion PDE
\begin{equation}\label{dd}
-\partial_t v-\eps\Delta v-b(x,t)\cdot Dv=0.
\end{equation}
Note that the knowledge of quantitative and qualitative properties of this latter equation is the lynchpin for the study of the high-regularity theory for many nonlinear PDEs. In fact, our main motivation stems from the analysis of time-dependent viscous Hamilton-Jacobi equations: if $u$ solves the viscous Hamilton-Jacobi equation
\begin{equation}\label{HJ}
-\partial_t u-\eps\Delta u+H(Du)=0,
\end{equation}
then any directional derivative $v=\partial_e u$ solves \eqref{dd} with $b(x,t)=-D_pH(Du(x,t))$. Also, the difference of two solutions solves a linearized equation of the form \eqref{dd} with a possibly different velocity field involving $D_pH$.\\

Existence and uniqueness of solutions for \eqref{fp}-\eqref{dd} are standard when the drift vector field is Lipschitz continuous, cf. \cite{LSU} for $\eps>0$ and \cite{CrippaThesis} for $\eps=0$. The well-posedness in a discontinuous setting has been a challenging domain of research, being mainly motivated by problems in fluid dynamics and stochastic control. It mainly proceeds through a main step, which consists in proving a suitable a priori estimate, combined with a fixed point argument and/or a regularization procedure. The former can be performed under certain conditions on the drift that can be in turn  understood focusing on the integral term involving the convection 
\begin{equation}
\label{eq:intbDrr}
\iint b\cdot D\rho\ \rho\,dxdt,
\end{equation}
where we used the solution $\rho$ as a test function in the weak formulation. There are four known main regimes under which one can establish the well-posedness of these linear PDEs with rough coefficients:
\begin{itemize}
\item[(a)] the Ambrosio-Diperna-Lions \cite{Ambrosio04,DipernaLions} setting, where \eqref{eq:intbDrr} (and hence $L^2$ a priori estimates) can be handled via an integration by parts as follows
\[
\int b\cdot D\rho\ \rho\,dx=\int b\cdot D\left(\frac{\rho^2}{2}\right)\,dx=-\int \mathrm{div}(b)\frac{\rho^2}{2}\,dx
\]
It is sufficient to ask $b$ bounded and the following conditions
\[
b\in L^1_t(W^{1,1}_x) \text{ or }b\in L^1_t(BV_x)\text{ and }[\mathrm{div}(b)]^-\in L^1_tL^\infty_x
\]
to conclude $L^2$ a priori estimate, via the Gronwall inequality, and the well-posedness of the equation \cite{CrippaThesis}. These estimates actually hold in the inviscid case $\eps=0$, but still the diffusive regime can be addressed for the so-called parabolic solutions, see e.g. \cite{BonicattoCiampaCrippa2022,BonicattoCiampaCrippa2024,LeBrisLions}.

\item[(b)] Lions-Seeger setting \cite{LionsSeegerARMA}, addressed within the framework of viscosity solutions \cite{CrandallIshiiLions}, in the case in which $b$ satisfies the one-sided Lipschitz condition given by 
\[
(b(x,t)-b(y,t))\cdot(x-y)\geq -C(t)|x-y|^2
\]
for some nonnegative $C \in L^1([0,T])$ or it has coordinate-by-coordinate semi-increasing velocity fields \cite{LionsSeegerAIHP}. 

\item[(c)] the Aronson-Serrin-Ladyzhenskaya setting \cite{AronsonSerrin,LSU,BOP,Bianchini2017}, where one does not proceed by integration by parts in \eqref{eq:intbDrr}. The H\"older inequality first and the (parabolic) Sobolev inequality then yield
\[
\iint b\cdot D\rho\ \rho\,dxdt\leq \|b\|_{L^{n+2}_{x,t}}\|D\rho\|_{L^2_{x,t}}\|\rho\|_{L^{\frac{2(n+2)}{n}}_{x,t}}\lesssim  \|b\|_{L^{n+2}_{x,t}}\|D\rho\|_{L^2_{x,t}}^2,
\]
which shows that one can deduce a priori $L^2$ estimates and the well-posedness of the equation in the parabolic class
\[
\mathcal{H}_2^1(Q_T):=\{u\in L^2(0,T;H^1(\Omega)),\ \partial_t u\in L^2(0,T;H^{-1}(\Omega))\}.
\]
\item[(d)] The other known conditions under which the well-posedness holds concern integrability conditions on the drift vector field or its derivatives against the solution, such as $b\in L^k(\rho\,dxdt)$ or $\mathrm{div}(b)\in L^k(\rho\,dxdt)$ for some $k>1$. These were deeply investigated in \cite{BKRS,PorrettaARMA,CirantGoffiAPDE,MetafuneJFA} and they have strong connections with stochastic control, being important for applications in Mean Field Games theory.
\end{itemize}
Our purpose in this paper is to investigate two new conditions under which a priori estimates and the uniqueness of solutions for \eqref{fp}-\eqref{dd} can be achieved, and apply them to the study of continuous dependence estimates and uniqueness of Hamilton-Jacobi equations. A common feature of our results is that our velocity field will not be divergence-free, which often appears and it is motivated by fluid dynamics settings.\\

First, we will investigate an intermediate regime that lies at the crossroad between (a) and (c) and exploits instead the regularizing effect of the heat operator. In fact, the term \eqref{eq:intbDrr}, after an integration by parts, gives
\[
\iint_{\T^n\times(0,T)} b\cdot D\rho\ \rho\,dxdt=-\iint_{\T^n\times(0,T)} \mathrm{div}(b)\frac{\rho^2}{2}\,dxdt
 \]
and thus can be also handled directly via the H\"older inequality and estimated by
\[
\iint_{\T^n\times(0,T)} \mathrm{div}(b)\frac{\rho^2}{2}\,dxdt\lesssim \|\mathrm{div}(b)\|_{L^{\frac{n+2}{2}}}\|\rho\|^2_{L^{\frac{2(n+2)}{n}}}
\]
Recall that parabolic Sobolev embeddings, cf. \cite{LSU}, yield
\[
\|\rho\|^2_{L^{\frac{2(n+2)}{n}}}\lesssim \|\rho\|_{\mathcal{H}_2^1}^2.
\]
The condition
\[
\mathrm{div}(b)\in L^{\frac{n+2}{2}}
\]
can be seen as intermediate among (a) and (c). In fact, on the one hand the order of integrability has been lowered with respect to (c) (where $b$ is asked to belong to the higher Lebesgue space $L^{n+2}$), but, on the other hand, we are requiring an integrability information on its derivatives (weaker than (a)). As discussed before, the responsible of this different integrability regime is the presence of the diffusion in the equation. We will prove in this and more general cases, such as the weakened Aronson-Serrin interpolated condition in mixed Lebesgue spaces
\begin{equation}\label{lowAS}
\mathrm{div}(b)\in L^r_t(L^q_x)\text{ with }\frac{n}{2q}+\frac{1}{r}\leq 1,
\end{equation}
$L^p$-stability and uniqueness of certain weak solutions. Some local regularity properties and maximum principles appeared under related conditions, assuming $b\in L^{q}_{x,t}$, $\frac{n+2}{2}<q\leq n+2$, (i.e. weaker than (c)) and a sign on the divergence of the velocity field \cite{NazarovUraltseva}. The integrability condition \eqref{lowAS} is natural, since \eqref{fp} is invariant under the scaling
\[
u_\lambda(x,t)=u(\lambda x,\lambda^2 t), \qquad  b_{\lambda}(x,t)=\lambda b(\lambda x,\lambda^2 t)
\]
and the space $L^r_t(L^q_x)$ is invariant for the scaling of $\mathrm{div}(b_\lambda)$ when $\frac{n}{2q}+\frac{1}{r}= 1$.\\

One possible application of the a priori regularity results in this setting concerns Mean Field Games systems, a model case being
\[
\begin{cases}
-\partial_t u(x,t)-\Delta u(x,t)+\frac{|Du(x,t)|^2}{2}=f(m(x,t))&\text{ in }\T^n\times(0,T),\\
\partial_t m(x,t)-\Delta m(x,t)-\mathrm{div}(Du(x,t)\ m(x,t))=0&\text{ in }\T^n\times(0,T),\\
u(x,t)=u_T(x),\ m(x,0)=m_0(x)&\text{ in }\T^n.
\end{cases}
\]
In the case of power-like couplings $f(m)=m^\alpha$, if we know that $m^\alpha\in L^q_{x,t}$ for some $q$ (i.e. $m\in L^{q/\alpha}$), then by maximal $L^q$-regularity for Hamilton-Jacobi equations $\Delta u\in L^q_{x,t}$ for some $q>(n+2)/2$, cf. \cite{CirantGoffiAPDE}. Setting $-b=Du$ and noting that $\mathrm{div}(b)=\Delta u$, we have that $\mathrm{div}(b)\in L^q_{x,t}$ for $q>(n+2)/2$. Our estimates, cf. Theorem \ref{main2}, will imply $m\in L^\infty_t(L^p_x)$ for all $p>1$, which in turn boosts the regularity of the right-hand side of the first equation. \\
We mention that the well-posedness of transport-diffusion equations for $b$ satisfying  \eqref{lowAS} was pointed out by P.-L. Lions in his lectures at Coll\'ege de France \cite{LionsSeminar}, see also \cite[Section 1.3.3]{LeBrisLions}, when $b=b(x)$ is independent of the time variable. Our main development is a full treatment of uniqueness and stability of the parabolic case in mixed Lebesgue scales. Many extensions are possible in this scenario: for instance, similar results hold for more general diffusions in divergence form $-\mathrm{div}(\sigma(x)\sigma^T(x)D\rho)$ for which a Sobolev inequality holds. This is the case of many hypoelliptic operators treated for instance in \cite[Section 3.5]{LeBrisLions}.

 The second condition on the velocity field we discuss is mainly related with the Ambrosio-Diperna-Lions setting (a) and it requires a one-side assumption on the velocity field which blows up in time: namely we consider continuous or Lipschitz solutions and impose
\[
\|[\mathrm{div}(b(t))]^-\|_{L^\infty(\T^n)}\lesssim \frac{c}{t}.
\]
This is reminiscent of Nagumo-type uniqueness criteria for ordinary differential equations \cite{Constantin10}, uniqueness and sup-norm contraction estimates for first-order, convex, Hamilton-Jacobi equations \cite{Douglis}, and reminds the Oleinik condition of conservation laws. In fact, the a priori estimates we derive do not depend on the viscosity $\eps$. A uniqueness result when the velocity field $b\in BV_x$ and whose $BV$ norm has a possible blow-up at $t=0$ can be found in \cite{ACFS}: this weaker condition can be imposed at the expenses of assuming an additional continuity of the solution in the time variable. The continuity of the solution in the time variable is essential in our argument too. We refer to \cite{GigaGiga} for a recent account on these topics. \\

We finally apply some of the estimates for drift-diffusion PDEs to address continuous dependence estimates for viscous Hamilton-Jacobi equations. These are crucial in the study of many quantitative properties for Hamilton-Jacobi PDEs, ranging from the regularity theory to vanishing viscosity and homogenization, cf. \cite{CaffarelliSouganidis}. To our knowledge, they started in \cite{SouganidisJDE} for first-order PDEs via doubling of variable methods. The derivation of properties for Hamilton-Jacobi equations by the study of their linearizations \eqref{dd} or their adjoint is not new, see e.g. \cite{Silvestre} and \cite{EvansARMA}, but they have never been applied to obtain continuous dependence estimates. Here we prove the following model estimate for solutions of \eqref{HJ} (even in the case when $H(x,t,p)$ depends on $x$) via duality methods: if $u_1,u_2$ are two solutions of \eqref{HJ} with right-hand side $f_1,f_2\in L^1_t L^{\infty}_x$ and terminal data $g_1,g_2 \in L^{\infty}$, we have
\begin{equation}\label{continuous}
\|(u_1-u_2)(t)\|_{L^{\infty}_{x,t}}\leq \|g_1-g_2\|_{L^\infty}+\int_0^t\|(f_1-f_2)(s)\|_{L^\infty_x}\,ds,\ t\in(0,T).
\end{equation}
The main feature of our estimate is that it is obtained without using the maximum principle. The novelties of our integral approach rely on the fact that all the constants are explicit and do not depend on the viscosity. It also allows to deduce new higher-order continuous dependence estimates and address $L^p$ bounds $\forall  \, p \in [1, +\infty]$. These latter scales are more familiar within the theory of scalar conservation laws and conservative equations \cite{Otto,Evans2007} and are not natural within the theory of viscosity solutions, which in turn is based on sup-norm techniques and the maximum principle. This also provides a new proof of uniqueness of solutions for Hamilton-Jacobi equations without exploiting the theorem on sums in the theory of viscosity solutions, a fact which we believe can be of independent interest. In particular, some consequences of our theoretical results on advection-diffusion PDEs are a uniqueness and stability property for solutions of Hamilton-Jacobi equations satisfying the condition 
\[
\Delta u\leq g(x,t)\in L^r_t(L^q_x)\text{ with }\frac{n}{2q}+\frac{1}{r}\leq 1,
\]
cf. Theorem \ref{th:ii}. These follow by duality via the stability result for Fokker-Planck equations under the condition \eqref{lowAS}, and address a question left open in \cite[Remark 3.6]{L82Book}; see Remark \ref{Lio}. This slightly weakens the requirement guaranteeing the uniqueness of solutions obtained by A. Douglis \cite{Douglis} and P. L. Lions \cite[Theorem 3.1]{L82Book}, where one needs $D^2u\leq (c_1/t+c_2)\mathbb{I}_n$ (semiconcavity for positive times) and $\Delta u\leq k$ respectively ($L^\infty$ semi-superharmonicity).
As a second progress, our approach is flexible enough to provide estimates of the form \eqref{continuous} for quite general problems posed on unbounded domains equipped with Neumann conditions, where few results are known (see e.g. \cite{Cockburnetal}), as well as in the whole space case. The technique has the potential to be applied for the derivation of continuous dependence estimates for degenerate diffusions (e.g. driven by subelliptic operators) or nonlocal operators (e.g. patterned over the fractional Laplacian) since the estimates depend weakly on the second order term. As a final contribution in this setting, we provide new continuous dependence and uniqueness results for time-dependent Hamilton-Jacobi equations with nonlinear terms having power-growth; see \cite{Barles} for recent developments via the theory of viscosity solutions.\\

\textit{Outline}. In Section \ref{sec;div1} we study $L^p$ stability estimates for Fokker-Planck equations under the condition $\mathrm{div}(b)\in L^r_t(L^q_x)$. In Section \ref{sec;div2} we address the case of transport-diffusion equations with $[\mathrm{div}(b(t))]^-\in L^\infty_x$ having a possible blow-up at $t=0$. In Section \ref{sec;HJ} we apply the previous results to obtain continuous dependence estimates for solutions of Hamilton-Jacobi equations.

\section{The case of vector fields with divergence in  mixed Lebesgue spaces}\label{sec;div1}
\subsection{$L^p$ a priori estimates and uniqueness}
Consider
\begin{equation}
\label{eq:fp}
\begin{cases}
\partial_t \rho-\Delta \rho+\mathrm{div}(b(x,t)\rho)=0 \quad & in \quad  Q_T:= \mathbb{T}^n \times (0,T), \\
\rho(x,0)=\rho_0(x) &in \quad  \mathbb{T}^n,
\end{cases}
\end{equation}
where $\T^n\equiv \R^n/ \Z^n$ and we set $\eps=1$, $n\geq2$, and let
\[
\mathcal{H}_2^1(Q_T):=\{u\in L^2(0,T;H^1(\T^n)),\ \partial_t u\in L^2(0,T;H^{-1}(\T^n))\}.
\]
Recall that the next estimates exploit the diffusion of the PDE and hence will also depend on the viscosity $\eps$.
\begin{thm}\label{stability}
Assume that
\begin{equation}\label{divb}
\mathrm{div}(b)\in L^r(0,T;L^q(\T^n)),
\frac{n}{2q}+\frac{1}{r}\leq 1,\text{ with } \begin{cases}
r\in[1,\infty)\text{ and }q\in(n/2,\infty]\text{ for }n\geq2\\
r\in[1,2]\text{ and }q\in[1,\infty]\text{ for }n=1.
\end{cases}
\end{equation}
Then every solution $\rho$ in the parabolic class $\mathcal{H}_2^1(Q_T)$ of \eqref{eq:fp} satisfies the estimate
\begin{equation}
\label{eq:L2estimate}
\|\rho(\cdot,t)\|_{L^2(\T^n)}\leq C_1 \|\rho_0\|_{L^2(\T^n)},
\end{equation}
where $C_1$ depends on $n,q,r,\|\mathrm{div}(b)\|_{L^r_t(L^q_x)}$. Morevoer, we have 
\begin{equation}
\label{eq:L2estgradient}
\int_0^T \int_{\mathbb{T}^n} |D\rho(x,t)|^2 dx dt \le C_2 \|\rho_0\|_{L^2(\T^n)},
\end{equation}
where $C_2$ depends on $n,q,r,\|\mathrm{div}(b)\|_{L^r_t(L^q_x)}, T$.
\end{thm}
\begin{proof}
We start by multiplying by $\rho$ the PDE and integrate in space $\T^n$. This implies
\begin{equation}
\label{eq:I}
\frac{d}{dt}\|\rho(t)\|_{L^2(\T^n)}^2+\int_{\T^n}|D\rho(t)|^2\,dx=-\int_{\T^n}\mathrm{div}(b)\frac{\rho^2}{2}\,dx=:I
\end{equation}
By the H\"older inequality
\[
I\leq \|\rho(t)\|_{L^{2q'}(\T^n)}^2\|\mathrm{div}(b(t))\|_{L^q(\T^n)}.
\]
We apply the Gagliardo-Nirenberg-Ladyzhenskaya inequality \cite[Remark 5, p.126]{Ni59} for $\theta \in (0,1)$ and find
\[
\|\rho(t)\|_{L^{2q'}(\T^n)}^2\|\mathrm{div}(b(t))\|_{L^q(\T^n)}\leq C_S\|\mathrm{div}(b(t))\|_{L^q(\T^n)}(\|D\rho(t)\|_{L^2(\T^n)}^{2\theta}\|\rho(t)\|_{L^2(\T^n)}^{2(1-\theta)}+\|\rho(t)\|_{L^2(\T^n)}^2)
\]
where
\begin{equation}\label{GN}
\frac{1}{2q'}=\frac12-\frac{\theta}{n}.
\end{equation}
The Young's inequality gives
\[
I\leq C_S (1-\theta)\left(\|\mathrm{div}(b(t))\|_{L^q(\T^n)}\|\rho(t)\|_{L^2(\T^n)}^{2(1-\theta)}\right)^{\frac{1}{1-\theta}}+\theta \|D\rho(t)\|_{L^2(\T^n)}^2+C_S\|\mathrm{div}(b(t))\|_{L^q(\T^n)}\|\rho(t)\|_{L^2(\T^n)}^2.
\]
Therefore
\begin{equation}
\label{eq:dis01}
\begin{aligned}
\frac{d}{dt}\|\rho(t)\|_{L^2(\T^n)}^2+(1-\theta)\int_{\T^n}|D\rho(t)|^2\,dx\leq  &(1-\theta)C_S\left(\|\mathrm{div}(b(t))\|_{L^q(\T^n)}\|\rho(t)\|_{L^2(\T^n)}^{2(1-\theta)}\right)^{\frac{1}{1-\theta}}\\
&+C_S\|\mathrm{div}(b(t))\|_{L^q(\T^n)}\|\rho(t)\|_{L^2(\T^n)}^2.
\end{aligned}
\end{equation}
Since the second term on the left-hand side is nonnegative, we get
\[
\frac{d}{dt}\|\rho(t)\|_{L^2(\T^n)}^2\leq  (1-\theta)C_S\|\mathrm{div}(b(t))\|_{L^q(\T^n)}^{\frac{1}{1-\theta}}\|\rho(t)\|_{L^2(\T^n)}^2+C_S\|\mathrm{div}(b(t))\|_{L^q(\T^n)}\|\rho(t)\|_{L^2(\T^n)}^2
\]
which gives the $L^2$ estimate  \eqref{eq:L2estimate} by the Gronwall inequality, provided that the following integrability on the velocity field holds
\begin{equation}
\label{eq:qrint}
\int_0^T\|\mathrm{div}(b(t))\|_{L^q(\T^n)}^{\frac{1}{1-\theta}}\,dt<\infty.
\end{equation}
Set 
\[
r= \frac{1}{1-\theta}
\]
and observe that by \eqref{GN}
\[
\frac{n}{2q}+\frac{1}{r} = \frac{n}{2}-\frac{n}{2q'}+1-\theta=1.
\]
Once we have that the estimate \eqref{eq:L2estimate}, going back to \eqref{eq:dis01} and integrating in time we get 
\begin{align*}
&(1-\theta) \int_0^T \int_{\T^n}|D\rho(t)|^2 \ dt \\
& \le  \left( \int_0^T   (1-\theta)C_S\|\mathrm{div}(b(t))\|_{L^q(\T^n)}^{\frac{1}{1-\theta}}+C_S\|\mathrm{div}(b(t))\|_{L^q(\T^n)} \ dt \right)   C \|\rho_0\|_{L^2(\T^n)}^2 .
\end{align*}
Finally by \eqref{eq:qrint} we obtain \eqref{eq:L2estgradient}.
\end{proof}

\begin{rem}
An inspection of the proof suggests that the condition $\mathrm{div}(b)\in L^r(0,T;L^q(\T^n))$ can be weakened to $ [\mathrm{div}(b)]^- \in L^r(0,T;L^q(\T^n))$ with  $\frac{n}{2q}+\frac{1}{r}\leq 1$. Indeed, going back to  the equation \eqref{eq:I} we have that  
\[
I\le \int_{\T^n}[\mathrm{div}(b)]^-  \frac{\rho^2}{2}\,dx.
\]
\end{rem}

\begin{cor}\label{uniquenessFP}
Assume that \eqref{divb} holds. If there exists  a solution $\rho$ in the parabolic class $\mathcal{H}_2^1(Q_T)$ of \eqref{eq:fp}, then $\rho$ is unique.
\end{cor}

\begin{thm}\label{main2}
Any nonnegative solution to \eqref{eq:fp} satisfies the estimate
\[
\|\rho(\cdot,t)\|_{L^{p}(\T^n)}\leq C\|\rho_0\|_{L^{p}(\T^n)},\ p\in(1,\infty),
\]
with $C$ depending on $n,q,r,p,\|\mathrm{div}(b)\|_{L^r_t(L^q_x)}$.
\end{thm}
\begin{proof}
It is enough to test the equation against $\rho^{2p-1}$, $p>1$, and proceed as above to get the inequality
\[
\frac{1}{2p}\frac{d}{dt}\||\rho(t)|^{p}\|_{L^2(\T^n)}^2+\int_{\T^n}|D\rho(t)|^2 (2p-1)\rho^{2(p-1)}\,dx\leq \int_{\T^n}|\mathrm{div}(b)|\frac{|\rho(t)|^{2p}}{2p}\,dx
\]
which gives
\[
\frac{d}{dt}\||\rho(t)|^p\|_{L^2(\T^n)}^2\leq \|\mathrm{div}(b(t))\|_{L^q(\T^n)}^{\frac{1}{1-\theta}}\||\rho(t)|^p\|_{L^2(\T^n)}^2+C_S\|\mathrm{div}(b(t))\|_{L^q(\T^n)}\||\rho(t)|^p\|_{L^2(\T^n)}^2. 
\]
This gives the $L^p$ stability for $p\geq2$. Since $\|\rho(t)\|_{1}=1$, by interpolation we get the estimate for any $p>1$.
\end{proof}

\begin{cor}
\label{cor:divLrLp}
Let $f \in L^1(\tau,T; L^p(\T^n))$, $p\in(1,\infty)$. Any solution to
\begin{equation}
\label{eq:dd}
\begin{cases}
-\partial_t v-\Delta v-b(x,t) \cdot Dv= f(x,t) \quad & in \quad  Q_T:= \mathbb{T}^n \times (\tau,T) \\
v(x,T)=v_T(x) &in \quad  \mathbb{T}^n,
\end{cases}
\end{equation}
with $v_T \in L^{p}(\mathbb{T}^n)$ such that \eqref{divb} holds satisfies the estimate
\begin{equation}
\label{eq:Lp}
\|v(\cdot,t)\|_{L^{p}(\T^n)}\leq C\left( \|v_T\|_{L^{p}(\T^n)} + \int_t^T \|f (x,s)\|_{L^p(\T^n)} ds \right),
\end{equation}
for all $t \in (\tau, T)$. 
\end{cor}
\begin{proof}
Let $1<p<+\infty$. Let $\rho$ be a non-negative solution of \eqref{eq:fp} with $\rho_\tau(x)\in L^{p'}$, $\rho_\tau\geq0$ and $\|\rho_\tau\|_{p'}\leq1$. Then using as a test function  $\rho$ in \eqref{eq:dd}, $v$ in \eqref{eq:fp} and subtracting  we gain 
\begin{equation}
\label{eq:corfe}
\frac{d}{dt} \int_{\mathbb{T}^n}  \rho(x,t) v(x,t) \ dx =\int_{\T^n} f(x,t) \rho(x,t) \ dx.
\end{equation}
Then we have via Theorem \ref{main2}
\begin{align*}
\int_{\mathbb{T}^n}  v(\tau)\rho(\tau) \ dx &=\int_{\mathbb{T}^n}  \rho(T) v(T) \ dx+  \int_{\tau}^T \int_{\T^n} f(x,t) \rho(x,t)   \ dx \\
&\le \|v(T)\|_{L^{p}} \|\rho(T)\|_{L^{p'}} +  \int_{\tau}^T  \|f(x,t)\|_{L^p(\T^n)} \|\rho(x,t)\|_{L^{p'}(\T^n)}   dt \\
&\le C\|v_T\|_{L^{p}}+ C  \int_{\tau}^T  \|f(x,t)\|_{L^p(\T^n)} dt.
\end{align*}
Moreover, we have
\[
\int_{\mathbb{T}^n}  \rho(T) v(T) \ dx\geq -C \|v_T\|_{L^{p}}
\]
and 
\[
  \int_{\tau}^T \int_{\T^n} f(x,t) \rho(x,t)   \ dx \geq -C  \int_{\tau}^T  \|f(x,t)\|_{L^p(\T^n)} dt.
\]
Coupling the previous terms and passing to the supremum over $\rho_\tau$ we get the statement \eqref{eq:Lp}.  
\end{proof}

\begin{rem}
\label{rm:AronsonSerrin}
On the other hand when  $b\in L^R(0,T;L^Q(\T^n))$ we are in the range $\frac{n}{2Q}+\frac{1}{R}\leq\frac12$ for $R\ge2$ and $Q>n$ by Theorem 2.1 and Remark 2.2 in \cite{Bianchini2017}, and we obtain the $L^p$ estimates as in Theorem \ref{main2}.
\end{rem}

\begin{rem}
The existence of solutions can be proved under further regularity conditions on $b$. If $b\in L^2(Q_\tau)$ one can prove the existence of entropy solutions \cite{BOP}. If, instead, $b\in L^{n+2}(Q_\tau)$ or, more generally, $b\in L^R(0,T;L^Q(\T^n))$ with $\frac{n}{2Q}+\frac{1}{R}\leq\frac12$ for $R\ge2$ and $Q\geq n$, the existence of energy solutions is well-known by \cite{DDaners2000,LSU}. The existence can be also achieved under rather different Sobolev assumptions on $b$, see \cite{LeBrisLions}.
\end{rem}

\section{Transport-diffusion equations with non-integrable one-side conditions on $\mathrm{div}(b)$}\label{sec;div2}
The next result shows a stability property for viscous transport equations under one-side conditions on $\mathrm{div}(b)$. This is well-known when $[\mathrm{div}(b)]^-\leq K(t)\in L^1(0,T)$, cf. \cite{CrippaThesis} and the references therein. Here we show that this can be achieved even when $[\mathrm{div}(b)]^-\lesssim K/t$. The next result is inspired by \cite{AronsonVazquez}.
\begin{thm}
\label{th:infinityest1t}
Let $\eps \ge 0$. Let $u_1,u_2$ be Lipschitz solutions of the viscous transport equation
\[
\partial_t u-\eps\Delta u-b(x,t)\cdot Du=f(x,t)\text{ in }Q_T:=\T^n\times(0,T)
\]
with initial conditions $g_1, g_2 \in L^{\infty}(\T^n)$ and $f\in L^1(0,T;L^\infty(\T^n))$, under the assumption that 
\begin{itemize}
\item $[\mathrm{div}(b)]^{-}\leq \dfrac{c_1}{t}+c_2 $,
\end{itemize}
where $c_1,c_2 \in \R$. Then
\[
\|(u_1-u_2)(t)\|_{L^\infty(\T^n)}\leq \|g_1-g_2\|_{L^\infty(\T^n)}.
\]
In particular, when $g_1=g_2$ we have  $u_1\equiv u_2$ on $Q_T$.\\
 If we have source terms $f_1, f_2\in L^1(0,T; L^{\infty}(\T^n))$ and $\|f_1- f_2\|_{L^1(0,T; L^{\infty}(\T^n))}>0$,  we have the estimate
\[
\|(u_1-u_2)(\tau)\|_{L^\infty(\T^n)}\leq  \exp \left( \int_{0}^{\tau} \|(f_1-f_2)(t)\|_{L^{\infty}(\T^n)} \ dt \right)  \left( \|g_1-g_2\|_{L^\infty(\T^n)}+1\right).
\]
\end{thm}
\begin{proof}
Take $w=(u_1-u_2)^+$ (hence $w(0)=u_1(0)-u_2(0)=g_1-g_2$) by \cite[Section 8]{LeonariPeralPrimoSoria}, \cite[Lemma 7.6]{GTbook} and by linearity
\[
\partial_t w-\eps\Delta w-b(x,t)\cdot Dw\leq \text{sign}^+ (w) (f_1-f_2)\text{ in }\T^n\times(0,T)
\]
Therefore, multiplying the equation by $w^{p-1}$ and integrating in space we get
\begin{align*}
\frac1p\frac{d}{dt}\int_{\T^n}w^p(t)\,dx&\leq -\frac{1}{p}\int_{\T^n}\mathrm{div}(b(t))w^p(t)\,dt + \|f_1-f_2\|_{L^{\infty}(\T^n)} \int_{\T^n} \text{sign}^+ (w) w^{p-1} \\
&\leq \frac{1}{p}\int_{\T^n} [\mathrm{div}(b(t))]^{-} w^p(t)\,dt + \|f_1-f_2\|_{L^{\infty}(\T^n)} \int_{\T^n} \text{sign}^+ (w) w^{p-1} \\
& \leq  \frac{1}{p}\left( \frac{c_1}{t}+c_2 \right) \int_{\T^n}w^p(t)\,dx + \|f_1-f_2\|_{L^{\infty}(\T^n)} \tfrac{p-1}{p} \int_{\T^n} w^p(t)\,dx+  \tfrac{1}{p}\|f_1-f_2\|_{L^{\infty}(\T^n)}\\
& \leq  \frac{1}{p}\left( \frac{c_1}{t}+c_2+ \|f_1-f_2\|_{L^{\infty}(\T^n)} (p-1) \right) \int_{\T^n}w^p(t)\,dx + \tfrac{1}{p}\|f_1-f_2\|_{L^{\infty}(\T^n)}.
\end{align*}
Integrating between $0<\sigma<\tau<T$ and setting $h(t)=\int_{\T^n}w^p(t)\,dx$ we get by the Gronwall inequality
\[
h(\tau)\leq \left(\frac{\tau}{\sigma}\right)^{c_1} \exp (c_2 T) \exp \left((p-1) \int_{\sigma}^{\tau} \|(f_1-f_2)(t)\|_{L^{\infty}(\T^n)} \ dt \right)  \,  \left( h(\sigma)+ \int_{\sigma}^{\tau} \|(f_1-f_2)(t)\|_{L^{\infty}(\T^n)} \ dt \right).
\]
Therefore
\begin{align*}
\left(\int_{\T^n}w^p(\tau)\,dx\right)^\frac1p&\leq \left(\frac{\tau}{\sigma}\right)^{\frac{c_1}{p}} \exp(c_2T)^{\frac{1}{p}}  \exp \left((p-1) \int_{\sigma}^{\tau} \|(f_1-f_2)(t)\|_{L^{\infty}(\T^n)} \ dt \right)^{\frac{1}{p}} \\& \quad \cdot \left(\int_{\T^n}w^p(\sigma)\,dx  +\int_{\sigma}^{\tau} \|(f_1-f_2)(t)\|_{L^{\infty}(\T^n)} \ dt\right)^\frac1p \\
&\leq \left(\frac{\tau}{\sigma}\right)^{\frac{c_1}{p}} \exp(c_2T)^{\frac{1}{p}}  \exp \left((p-1) \int_{\sigma}^{\tau} \|(f_1-f_2)(t)\|_{L^{\infty}(\T^n)} \ dt \right)^{\frac{1}{p}} \\& \quad \cdot \left( \|w( \sigma)\|_p + \left(\int_{\sigma}^{\tau} \|(f_1-f_2)(t)\|_{L^{\infty}(\T^n)} \ dt \right)^\frac1p  \right), \\
\end{align*}
since $(a+b)^{\alpha} \le a^{\alpha}+ b^{\alpha}$ for $a, b >0$ and $0<\alpha<1$.
Letting $p\to+\infty$ we conclude
\[
\|(u_1-u_2)^+(\tau)\|_{L^\infty(\T^n)}\leq  \exp \left( \int_{\sigma}^{\tau} \|(f_1-f_2)(t)\|_{L^{\infty}(\T^n)} \ dt \right) \left(  \|(u_1-u_2)^+(\sigma)\|_{L^\infty(\T^n)} + 1 \right),
\]
when $\int_0^{\tau}  \|(f_1-f_2)(t)\|_{L^{\infty}(\T^n)} \ dt >0$ and 
\[
\|(u_1-u_2)^+(\tau)\|_{L^\infty(\T^n)}\leq    \|(u_1-u_2)^+(\sigma)\|_{L^\infty(\T^n)} ,
\]
when $f_1=f_2$.
Since $u_1,u_2$ are continuous, we get 
\[
\|(u_1-u_2)^+(\tau)\|_{L^\infty(\T^n)}\leq  \exp \left( \int_{0}^{\tau} \|(f_1-f_2)(t)\|_{L^{\infty}(\T^n)} \ dt \right)  \left( \|(u_1-u_2)(0)\|_{L^\infty(\T^n)}+1\right) 
\]
for $f_1 \neq f_2$ and 
\[
\|(u_1-u_2)^+(\tau)\|_{L^\infty(\T^n)}\leq   \|(u_1-u_2)(0)\|_{L^\infty(\T^n)} 
 \]
when $f_1=f_2$.
On the other hand, take now $w=(u_1-u_2)^-$ (hence $w(0)=u_1(0)-u_2(0)$), since $(u_1-u_2)^-=(u_1-u_2)^+-(u_1-u_2)$ and $u_1$ and $u_2$ are solutions by linearity we still have
\[
\partial_t w-\eps\Delta w-b(x,t)\cdot Dw\leq \left( \text{sign}^+ (w)-1\right) (f_1-f_2) \quad \text{ in }\T^n\times(0,T)
\]
Therefore, let $0<\sigma<\tau<T$ and repeating the previous argument and noticing that $\text{sign}^+ (w)-1 \le 0 <1$  we obtain
\[
\|(u_1-u_2)^-(\tau)\|_{L^\infty(\T^n)}\leq  \exp \left( \int_{0}^{\tau} \|(f_1-f_2)(t)\|_{L^{\infty}(\T^n)} \ dt \right)  \left( \|(u_1-u_2)(0)\|_{L^\infty(\T^n)}+1\right),
\]
when $f_1 \neq f_2$ and 
\[
\|(u_1-u_2)^-(\tau)\|_{L^\infty(\T^n)}\leq \|(u_1-u_2)(0)\|_{L^\infty(\T^n)}
\]
when $f_1 =f_2$.
Hence, the supremum norms of $(u_1-u_2)^+$ and  $(u_1-u_2)^-$ are bounded by $\|g_1-g_2\|_{L^\infty(\T^n)}$  we get
\[
\|(u_1-u_2)(\tau)\|_{L^\infty(\T^n)}\leq  \exp \left( \int_{0}^{\tau} \|(f_1-f_2)(t)\|_{L^{\infty}(\T^n)} \ dt \right)  \left( \|(g_1-g_2 \|_{L^\infty(\T^n)}+1\right) 
\]
when $f_1 \neq f_2$ and 
\[
\|u_1-u_2\|_{L^\infty(\T^n)} \leq  \|g_1-g_2 \|_{L^\infty(\T^n)},
\]
when $f_1=f_2$.
\end{proof}

\begin{rem}
Note that the bounds in Theorem \eqref{th:infinityest1t} do not depend on $\eps$. We can also weaken the regularity requirement on $u_1,u_2$ to the mere continuity in space-time, as in \cite{AronsonVazquez}. Note that the continuity in the time variable is essential in the previous argument, as in the work \cite{ACFS} which allows to deal with $\mathrm{BV}_x$ vector fields whose $\mathrm{BV}_x$-norm has a possible blow-up at $t=0$.
\end{rem}

\section{Applications to uniqueness and continuous dependence estimates of Hamilton-Jacobi equations}\label{sec;HJ}
We now show how to apply some of the previous integrability estimates for Fokker-Planck equations to prove continuous dependence estimates and uniqueness results for Hamilton-Jacobi equations without employing viscosity solutions' theory.\\

We consider the backward viscous Hamilton-Jacobi equation
\begin{equation}
\label{eq:HJ}
\begin{cases}
-\partial_t u-\eps\Delta u+H(x,t,Du)=f(x,t) \quad & in \quad  Q_T:= \Omega \times (0,T) \\
u(x,T)=u_T(x) &in \quad  \Omega,
\end{cases}
\end{equation}
where the problem is either posed on $\T^n$ or on a regular domain $\Omega\subset\R^n$ and \eqref{eq:HJ} is equipped with the homogeneous Neumann boundary condition
\[
\partial_\nu u=0\text{ on }\partial\Omega\times(0,T).
\]
Note that $\Omega$ can be unbounded and even the whole space $\R^n$. 
We prove that if $u_1,u_2$ are two energy solutions of \eqref{eq:HJ} with different initial conditions $g_1,g_2$ and source terms $f_1,f_2$, we have
\[
\|(u_1-u_2)(t)\|_{L^\infty(\Omega)}\leq \|g_1-g_2\|_{L^\infty(\Omega)}+\int_0^T\|f_1(t)-f_2(t)\|_{L^\infty(\Omega)}\,dt.
\]
This type of estimates is often obtained applying the comparison principle to
\[
u_1\text{ and }u_2+\int_0^t\|(f_1-f_2)(s)\|_{\infty}\,ds+\|g_1-g_2\|_{\infty}.
\]
Our main concern here is to propose a new approach that avoids the use of maximum principle methods.
Moreover, we note that no convexity of $H$ is required to run the argument, and the above estimate does not depend on $\eps$. 
Furthermore, our solutions can be weak in the $L^2$ sense. Their existence is addressed for instance in \cite{CirantGoffi2020,PorrettaPriola} under certain a priori bounds on the solutions and weaker conditions on the right-hand side of the equation. Our first main result exploits the conservation of mass property for the dual equation of the linearization of the HJ equation. In the case of unbounded domains $\Omega\subset\R^n$ equipped with the Neumann condition, this holds when $\Omega$ satisfies the interior cone property and the drift $b(x,t)=-D_pH(x,t,Du)$ is bounded (as in Section \ref{Si}) or satisfies (locally) an Aronson-Serrin interpolated condition, combined with suitable conditions at infinity, as in Section \ref{Sii}. This is discussed in \cite[Theorem 2.4 and Remark 2.5]{GoffiTralli}.
\begin{thm}\label{th:hjlip}
 Let $u_1,u_2$ be two solutions of \eqref{eq:HJ} with terminal data $g_1,g_2\in L^\infty(\Omega)$ and $f_1,f_2\in L^1(0,T;L^\infty(\Omega))$, and $D_pH\in L^{\infty}_{\mathrm{loc}}(Q_T)$. Then for all $t\in(0,T)$ we have
 \[
\|(u_1-u_2)(t)\|_{L^\infty(\Omega)}\leq \|g_1-g_2\|_{L^\infty(\Omega)}+\int_0^T\|f_1(t)-f_2(t)\|_{L^\infty(\Omega)}\,dt.
\]
\end{thm}

\begin{proof}
Let $w=u_1-u_2$ and note that it satisfies the evolution PDE
\[
-\partial_t w-\eps\Delta w+b(x,t)\cdot Dw=f_1(x,t)-f_2(x,t)\text{ in }\Omega\times(0,T)
\]
equipped with $w(x,T)=g_1(x)-g_2(x)$ and homogeneous Neumann conditions (we can neglect the boundary values for manifolds without boundary like $\T^n$ or the whole $\R^n$), where
\[
b(x,t)=\int_0^1D_pH(x,t,\theta Du_1+(1-\theta)Du_2)\,d\theta.
\]
Using the solution $\rho\in \mathcal{H}_2^1(Q_T)$ of the (forward) adjoint problem
\begin{equation*}
\begin{cases}
\partial_t \rho-\eps\Delta \rho-\mathrm{div}(b(x,t)\rho)=0 \quad & in \quad  Q_\tau:= \Omega \times (\tau,T), \\
\partial_\nu \rho+\rho b\cdot \nu=0 &in \quad  \partial\Omega\times(0,T),\\
\rho(x,\tau)=\rho_\tau(x) &in \quad  \Omega,
\end{cases}
\end{equation*}
which exists thanks to \cite[Theorem 2.4]{GoffiTralli}, as a test function in the weak formulation of the PDE solved by $w$ we get, by duality,
\[
\int_\Omega w(\tau)\rho(\tau)\,dx=\int_\Omega w(T)\rho(T)\,dx+\int_\tau^T\int_\Omega (f_1(t)-f_2(t))\rho\,dxdt.
\]
Using the conservation of mass $\int_\Omega \rho(s)\,dx= 1$ we get the statement by the H\"older inequality.
\end{proof}

\begin{rem}
Notice that all the results of this section hold replacing $\eps \Delta u $ with non-homogeneous diffusions in non-divergence form driven by $\text{Tr}(A(x,t) D^2 u(x,t))$ or, more in general, $F(x,t,D^2u)$ where $F$ is uniformly elliptic. Indeed $w=u_1-u_2$ solves the
\[
-\partial_t w- a_{ij}(x,t) D_{i,j}^2w+b(x,t)\cdot Dw=f_1(x,t)-f_2(x,t)\text{ in }\Omega\times(0,T),
\]
where 
\[
a_{ij}(x,t)= \int_0^1 F_{i,j} (x,t, \theta D^2u_1+(1-\theta)D^2u_2) \, d\theta.
\]
Note that $(a_{ij})$ is well-defined and uniformly elliptic since the uniform ellipticity of $F$ implies the Lipschitz continuity in the matrix entry. In this case the adjoint equation becomes
\[
\partial_t \rho- D_{i,j}^2 (a_{i,j}(x,t)\rho)-\mathrm{div}(b(x,t)\rho)=0.
\]
\end{rem}

\begin{rem}
The above result hides some further regularity assumptions on $H$ and $u$. If $\Omega$ is a bounded domain the drift $b\sim D_pH$ is globally bounded and we do not need any further assumption. If $\Omega$ is unbounded we need Lipschitz estimates to have the drift vector field $b$ globally bounded on the whole domain. The proof of Lipschitz estimates typically requires coercivity type conditions of $H$ in the gradient entry, but not necessarily the convexity. Clearly, if either $H$ is strictly convex (e.g. with power growth in the gradient) or uniformly convex, these Lipschitz bounds are well-known. Lipschitz bounds do not necessarily depend on $\eps$, see \cite{LioAA, Lioduke}.
\end{rem}

\begin{rem}
Notice that in the special case $H(x,t,p)=h(|p|)$, where $h:\R^+\to\R$ is a $C^1$ function, we have
\[
b(x,t)\cdot \nu=\left(\int_0^1D_pH(x,t,\theta Du_1+(1-\theta)Du_2)\,d\theta\right)\cdot \nu=0,
\]
and the boundary condition $\partial_\nu\rho=0$ is sufficient.
\end{rem}

\begin{rem}[Homogeneous Dirichlet boundary conditions]
The above proof can be repeated when the problem is equipped with homogeneous Dirichlet boundary conditions
\[
u(x,t)=0\text{ on }\partial\Omega\times(0,T).
\]
If the domain is bounded we have $\int_\Omega \rho(s)\,ds\leq1$ via the divergence theorem (it is enough to test the PDE solved by $\rho$ against $\varphi=1$). In the case of domains with non-compact boundary we expect a similar property, but we are not aware of any result in this direction.
\end{rem}

In the next sub-Sections we address three cases where estimates are \textit{independent of the viscosity $\eps$} through the results of the previous sections:
\begin{itemize}
\item[(i)] The first case is that of Lipschitz continuous solutions satisfying one of the following semiconcavity-type conditions
\begin{equation}\label{semic1t}
D^2u \leq\left(\frac{c_1}{t}+c_2\right)\mathbb{I}_n.
\end{equation}
This leads to a Douglis-type estimate for viscous HJ equations.
\item[(ii)] The second case is the ``weak'' Aronson-Serrin interpolated condition in mixed Lebesgue spaces 
	\[
	\mathrm{div}(D_pH)\in L^r(0,T;L^q(\T^n)) \qquad \text{such as} \quad  \frac{n}{2q}+\frac{1}{r}\leq 1.
	\]
\item[(iii)] The third is that of weak energy solutions in $\mathcal{H}_2^1\cap C(\overline{Q}_T)$ satisfying the integrability condition
\begin{equation}
\label{eq:PQ12}
D_pH\in L^R(0,T;L^Q(\T^n)),\quad\frac{n}{2Q}+\frac{1}{R}\leq\frac12,
\end{equation}
where $R\ge2$ and $Q>n$.
\end{itemize}

\subsection{Semiconcavity-type condition (i)} \label{Si} We start by recalling some well-known uniqueness properties of the first-order Hamilton-Jacobi equation \cite{Benton}.  The 2D Hamilton-Jacobi equation
\[
-\partial_t u+u_x^2=0\text{ in }\R\times\R
\]
admits more than one global solution. For instance, $u_1(x,t)=0$ is a solution (with Cauchy datum $u_1(x,0)=0$), $u_2(x,t)=|x|+t$ is another solution, and 
\[
u_3(x,t)=
\begin{cases}
|x|+t&\text{ if }|x|\leq t,\\
0&\text{ if }|x|\geq t
\end{cases}
\]
is again a solution (starting with the initial datum $u_3(x,0)=0$). Note that $u_1$ is semiconcave, while $u_3$ is not semiconcave in the space variable. Well-known results in the theory of first-order Hamilton-Jacobi PDEs establish that the semiconcavity assumption is sufficient to single out the correct physical solution and prove a selection principle for uniformly convex Hamilton-Jacobi PDE. \\

Motivated by this example, we provide a similar criterion for viscous Hamilton-Jacobi equations, showing that for some special convex Hamiltonians the requirement guaranteeing the uniqueness can be considerably weakened using PDE methods, see Theorem \ref{unique-ssh}. Let us consider 
\begin{equation}
\label{eq:HJ2}
\begin{cases}
-\partial_t u-\eps\Delta u+H(x,t,Du)=f(x,t) \quad & in \quad  Q_T:= \T^n \times (0,T) \\
u(x,T)=g(x) &in \quad  \T^n.
\end{cases}
\end{equation}

\begin{thm}
Let $u_1,u_2$ be two solutions of \eqref{eq:HJ2} as in (i) (see \eqref{semic1t}) such that the terminal data $g_1,g_2\in L^\infty(\T^n)$ and source terms  $f_1, f_2\in L^1(0,T;L^\infty(\T^n))$. Assume there exists $\lambda, \Lambda \in \R$ such that 
\begin{equation}
\label{eq:Helliptic}
\lambda |\xi|^2 \le D_{p_ip_j}^2 H(x,t, p) \xi_{i} \xi_{j} \le \Lambda |\xi|^2
\end{equation} 
for each $\xi \in \R^n$ and $D_p H \in W^{1, \infty}_{\mathrm{loc}}(Q_T)$.
If $\|f_1- f_2\|_{L^1(0,T; L^{\infty}(\T^n))}>0$, then for all $t\in(0,T)$ we have
\begin{equation}
\label{thm2:f}
\|(u_1-u_2)(t)\|_{L^\infty(\T^n)}\leq  \exp \left( \int_{0}^{\tau} \|(f_1-f_2)(t)\|_{L^{\infty}(\T^n)} \ dt \right)  \left( \|g_1-g_2\|_{L^\infty(\T^n)}+1\right) 
\end{equation}
Moreover when $f_1=f_2$  we get
\begin{equation}
\label{thm2:0}
\|(u_1-u_2)(t)\|_{L^\infty(\T^n)}\leq \|g_1-g_2\|_{L^\infty(\T^n)}.
\end{equation}
\end{thm}
\begin{proof}
Let $w=u_1-u_2$ and note that it satisfies the evolution PDE
\[
-\partial_t w-\eps\Delta w-b(x,t)\cdot Dw=f_1-f_2\text{ in }\T^n \times(0,T)
\]
equipped with $w(x,T)=g_1(x)-g_2(x)$, where
\[
b(x,t)=-\int_0^1D_pH(x,t,\theta Du_1+(1-\theta)Du_2)\,d\theta.
\]
Then, $z(x,t)=w(x,T-t)$ solves the forward equation
\[
\partial_t z-\eps\Delta z-b(x,T-t)\cdot Dz=f_1(x,T-t)-f_2(x,T-t)\text{ in }\T^n \times(0,T)
\]
equipped with $z(x,0)=g_1(x)-g_2(x)$. Setting
\[
Du_\theta:=\theta Du_1+(1-\theta)Du_2.
\]
we  notice that 
\begin{equation}
\label{eq:divb}
\begin{aligned}
-\mathrm{div}(b(x,t))&= \int_0^1 \sum_{i=1}^n D_{x_i} D_{p_i} H(x,t, Du_\theta)\,d\theta+ \\
&\quad + \int_0^1 \sum_{i=1}^n \sum_{j=1}^n \theta D_{p_j, p_i}^2 H (x,t, Du_\theta) \, D_{i,j}^2 u_1 \ d\theta \\
&\quad  + \int_0^1 \sum_{i=1}^n \sum_{j=1}^n (1-\theta) D_{p_j, p_i}^2 H (x,t, Du_\theta) \, D_{i,j}^2 u_2  \ d \theta\\
&\le C_2 +  \int_0^1 \theta \text{Tr}(D^2 H D^2 u_1) \, d\theta +  \int_0^1 (1-\theta) \text{Tr}(D^2 H D^2 u_2) \, d \theta  \\
&\le C_2 + \frac{C_1}{t},
\end{aligned}
\end{equation}
where $C_1$ and $C_2$ depend on $\Lambda, c_1, c_2$ defined in the assumptions \eqref{semic1t} and \eqref{eq:Helliptic}.
Then, since  the trivial solution $v\equiv0$ satisfies $
\partial_t u-\eps\Delta u-b(x,t)\cdot Du=0 \quad \text{ in }\T^n \times(0,T)$  by Theorem \ref{th:infinityest1t} we have get \eqref{thm2:f} if $f_1 \neq f_2$ and \eqref{thm2:0} if $f_1 =f_2$.
\end{proof}
\begin{rem}
Some remarks on this uniqueness result are in order. In the theory of ODEs \cite{Constantin10} the growth condition of the velocity field must be of order exactly $c/t$, $c\leq1$ as $t\to0^+$ (in the case it behaves as $c/t$, $c>1$, there are counterexample to uniqueness of solutions, cf. \cite{Perron}). The conditions for the validity of uniqueness can be relaxed if one assumes further regularity on the velocity field: this is known as Krasnoel'skii-Kreine's uniqueness, cf. \cite{ODEbook}. Here the fact that $b$ stems from the Hamiltonian of a nonlinear HJ equation, i.e. $b\sim D_pH(Du)$, with $H$ having some convexity or coercivity properties, provides a certain degree of regularity which ensures a uniqueness condition for general $c>0$. This type of uniqueness result in the theory of first-order HJ equations goes back to A. Douglis \cite{Douglis}. \\

\end{rem}

The next result provides new uniqueness and continuous dependence properties of superquadratic diffusive Hamilton-Jacobi equations under the weaker semi-superharmonic condition $\Delta u\leq c$.
\begin{thm}\label{unique-ssh}
Let $u_1,u_2$ be two Lipschitz solutions of \eqref{eq:HJ2} satisfying $\Delta u_i\leq C$. Assume that the terminal data $g_1,g_2\in L^\infty(\T^n)$ and source terms  $f_1, f_2\in L^1(0,T;L^\infty(\T^n))$. Assume that 
\begin{equation}
\label{eq:Hpower}
H(p)=(1+|p|^2)^{\frac{\gamma}{2}},\ \gamma>1.
\end{equation} 
If $\|f_1- f_2\|_{L^1(0,T; L^{\infty}(\T^n))}>0$, then for all $t\in(0,T)$ we have
\begin{equation}
\label{thm2:f}
\|(u_1-u_2)(t)\|_{L^\infty(\T^n)}\leq  \exp \left( \int_{0}^{\tau} \|(f_1-f_2)(t)\|_{L^{\infty}(\T^n)} \ dt \right)  \left( \|g_1-g_2\|_{L^\infty(\T^n)}+1\right) 
\end{equation}
Moreover when $f_1=f_2$  we get
\begin{equation}
\label{thm2:0}
\|(u_1-u_2)(t)\|_{L^\infty(\T^n)}\leq \|g_1-g_2\|_{L^\infty(\T^n)}.
\end{equation}
\end{thm}

\begin{proof}
Note that
\[
D_pH(p)=\gamma(1+|p|^2)^{\frac{\gamma}{2}-1}p.
\]
Set 
\[
Du_\theta:=\theta Du_1+(1-\theta)Du_2.
\]
Back to \eqref{eq:divb} we have
\[
-\mathrm{div}(b(x,t))=\int_0^1[\gamma(\gamma-2)(1+|Du_\theta|^2)^{\frac{\gamma}{2}-2}|Du_\theta|^2+\gamma(1+|Du_\theta|^2)^{\frac{\gamma}{2}-1}(\theta \Delta u_1+(1-\theta)\Delta u_2)]\,d\theta.
\]
When $\gamma\geq2$, since $u_i$ are Lipschitz and $\Delta u_i\leq C$, we have $-\mathrm{div}(b(x,t))\leq K$. If $1<\gamma\leq2$, the first term is nonpositive, while the second one can be rewritten as
\[
\gamma\underbrace{\frac{1}{(1+|Du_\theta|^2)^{\frac{2-\gamma}{2}}}}_{\leq 1\text{ since }\gamma\in(1,2)}(\theta \Delta u_1+(1-\theta)\Delta u_2)\leq K. 
\]
\end{proof}

\begin{rem}
The proof of the continuous dependence estimate in case (i) when $H(x,t,\cdot)\in W^{1,\infty}$  hides some convexity type conditions to deduce $D^2u\leq c(t)$. However, these bounds can be achieved even for Hamiltonians not necessarily uniformly convex, cf. \cite{FlemingJDE}.\\
It is well-known that uniqueness criteria for generalized solutions in $W^{1,\infty}$ of first-order HJ equations introduced by Douglis and Kruzhkov require $H$ to be convex and a one-side bound on second order derivatives \eqref{semic1t}. This latter requirement was weakened by P.-L. Lions \cite{L82Book} for semisuperharmonic solutions of first-order stationary HJ equations, i.e. assuming $\Delta u\leq c$ (a control on the trace) and some degree of regularity of solutions. It is also well-known that Lipschitz semiconcave solutions of convex Hamilton-Jacobi equations are also viscosity solutions of the same equation, and the uniqueness can be deduced by classical viscosity methods in \cite{CrandallIshiiLions}. Our first novelty is that $H$ needs not be convex in the gradient variable and thus our result cannot be directly inherited from Kruzhkov-Douglis results. Thus we provide a non-convex setting where one can deduce quantitative uniqueness of solutions without using the viscosity solutions' theory. Moreover, we provide for the first time a quantitative bound for gradients in $L^p$ spaces: these estimates are known to be difficult to retrieve from the known techniques in the theory of viscosity solutions, cf. \cite{EvansARMA}. This can be of interest for homogenization theory, where continuous dependence estimates play a crucial role.
\end{rem}

In the setting (i), we are also able to prove higher-order continuous dependence estimates for the gradient of solutions of \eqref{eq:HJ}. We are not aware of any other quantitative continuous dependence estimate involving the gradient of solutions of Hamilton-Jacobi equations. 
\begin{cor}  Let $u_1,u_2$ be two solutions of \eqref{eq:HJ} as in Theorem \ref{th:hjlip}. Then we have
\[
\|(Du_1-Du_2)(t)\|_{L^2(\T^n)}^2\leq \|\Delta(u_1-u_2)\|_{L^1(Q_T)}\left(\|g_1-g_2\|_{L^\infty(\T^n)}+\int_0^T\|f_1(t)-f_2(t)\|_{L^\infty(\T^n)}\,dt\right)
\]
\end{cor}
\begin{proof}
\begin{multline*}
\|(Du_1-Du_2)(t)\|_{L^2(\T^n)}^2=\int_{\T^n}|Du_1(t)-Du_2(t)|^2\,dx=-\int_{\T^n}(u_1-u_2)(t)\Delta(u_1-u_2)(t)\,dx\\
\leq \|(u_1-u_2)(t)\|_{L^\infty(\T^n)}\|\Delta(u_1-u_2)(t)\|_{L^1(\T^n)}.
\end{multline*}
Using the continuous dependence bound in Theorem \ref{th:hjlip} with $\Omega=\T^n$
\[
\|(Du_1-Du_2)(t)\|_{L^2(\T^n)}^2\leq K\left(\|g_1-g_2\|_{L^\infty(\T^n)}+\int_0^T\|f_1(t)-f_2(t)\|_{L^\infty(\T^n)}\,dt\right)
\]

\end{proof}

\begin{rem}
Note that $\Delta u\in L^1$ holds when $u$ is semiconcave and the equation is posed on a bounded domain, see \cite{L82Book}.
\end{rem}

\begin{rem}
We can get further first-order continuous dependence estimates for the term $\|(Du_1-Du_2)(t)\|_{L^p(\T^n)}$ via Gagliardo-Nirenberg inequalities.
\end{rem}

In this setting we can also achieve new continuous dependence estimates in $L^1$, which are typical scales for stability estimates of conservation laws. In fact, we have
\begin{thm}
Let $u_1,u_2$ be two solutions of \eqref{eq:HJ2} as in (i) and satisfying $D^2u(x,t)\leq c(t)\in L^1(0,T)$ such that the terminal data $g_1,g_2\in L^1(\T^n)$ and source terms  $f_1, f_2\in L^1(Q_T)$. Assume there exists $\lambda, \Lambda \in \R$ such that 
\begin{equation}
\label{eq:Helliptic}
\lambda |\xi|^2 \le D_{p_ip_j}^2 H(x,t, p) \xi_{i} \xi_{j} \le \Lambda |\xi|^2
\end{equation} 
for each $\xi \in \R^n$. Then
\[
\|u_1-u_2\|_{L^\infty(0,T;L^1(\T^n))}\leq \exp\left(\int_\tau^T K(t)\,dt\right)\left(\|g_1-g_2\|_{L^1(\T^n)}+\|f_1-f_2\|_{L^1(Q_\tau)}\right)
\]
for some $K\in L^1(0,T)$.
\end{thm}
\begin{proof}
Start again with the PDE solved by $w=u_1-u_2$, namely
\[
-\partial_t w-\eps\Delta w-b(x,t)\cdot Dw=f_1-f_2\text{ in }\T^n \times(0,T)
\]
equipped with $w(x,T)=g_1(x)-g_2(x)$, where
\[
b(x,t)=-\int_0^1D_pH(x,t,\theta Du_1+(1-\theta)Du_2)\,d\theta.
\]
By duality with the Fokker-Planck equation \eqref{eq:fp} with initial condition $\rho(x,\tau)=\mathrm{sgn}(w)$, by equation \eqref{eq:divb} and assuming the condition $D^2u(x,t)\leq c(t)\in L^1(0,T)$, we find that
\[
-\mathrm{div}(b(x,t))\leq K(t)\in L^1_t.
\]
By the maximum principle estimates we have
\[
\|\rho\|_{L^\infty(Q_\tau)}\leq \exp\left(\int_\tau^T K(t)\,dt\right)\|\rho(\tau)\|_{L^\infty(\T^n)}.
\]
Thus, by the integral identity
\[
\int_{\T^n}w(\tau)\rho(\tau)\,dx=\int_{\T^n}w(T)\rho(T)\,dx+\iint_{Q_\tau}(f_1-f_2)\rho\,dxdt
\]
and using the choice $\rho(x,\tau)=\mathrm{sgn}(w)$ we conclude
\[
\int_{\T^n}|w(\tau)|\,dx\leq \exp\left(\int_\tau^T K(t)\,dt\right)\left(\|g_1-g_2\|_{L^1(\T^n)}+\|f_1-f_2\|_{L^1(Q_\tau)}\right).
\] 
\end{proof}

\subsection{The ``weak'' Aronson-Serrin interpolated condition (ii)} \label{Sii}
\begin{thm}
\label{th:ii}
Let $u_1,u_2$ be two solutions of \eqref{eq:HJ} with $[\mathrm{div}(D_pH)]^+\in L^r(0,T;L^q(\T^n))$ such that $\frac{n}{2q}+\frac{1}{r}\leq 1$, terminal data $g_1,g_2\in L^p(\T^n)$ and $f_1,f_2\in L^1(0,T;L^p(\T^n))$. Then for all $t\in(0,T)$ 
\[
\|(u_1-u_2)(t)\|_{L^p(\T^n)}\leq C\left(\|g_1-g_2\|_{L^p(\T^n)}+\int_0^T\|f_1(s)-f_2(s)\|_{L^p(\T^n)}\,ds\right).
\]
\end{thm}
\begin{proof}
Let $w=u_1-u_2$ and note that it satisfies the evolution PDE
\[
-\partial_t w-\eps\Delta w-b(x,t)\cdot Dw=f_1-f_2\text{ in }\T^n \times(0,T)
\]
equipped with $w(x,T)=g_1(x)-g_2(x)$, where
\begin{equation}
\label{eq:bxt}
b(x,t)=-\int_0^1D_pH(x,t,\theta Du_1+(1-\theta)Du_2)\,d\theta.
\end{equation}
Since $[\mathrm{div}(D_pH)]^+\in L^r(0,T;L^q(\T^n))$ and by the Jensen's inequality we gain that $[\mathrm{div}(b(x,t))]^-$ belongs to $L^r(0,T;L^q(\T^n))$. Thus, thanks to \eqref{eq:Lp} in Corollary \ref{cor:divLrLp} we get the desired result.
\end{proof}

\begin{rem}\label{Lio}
The condition $[\mathrm{div}(D_pH)]^+\in L^r_t(L^q_x)$ is satisfied, for instance, when $H$ is uniformly convex in $p$ and $D^2u(x,t)\leq c(x,t)\mathbb{I}_n,\ c(x,t)\in L^r_t(L^q_x)$ with $r,q$ as above. This weakens the classical semiconcavity assumption of a result by  A. Douglis \cite{Douglis}. In the special case $H(Du)=\frac{|Du|^2}{2}$ we have $\mathrm{div}(D_pH)=\Delta u$. Therefore, we can conclude $L^p$ stability estimates and uniqueness when $u_1,u_2$ are solutions of Hamilton-Jacobi equations satisfying $\Delta u_i\leq g(x,t)\in L^r_t(L^q_x)$ with $r,q$ as above. This answers a uniqueness problem left open in \cite[Remark 3.6 p.87]{L82Book} and weakens the classical semi-superharmonicity condition ($\Delta u\leq k$) used to prove the uniqueness of solutions in \cite[Theorem 3.1]{L82Book}.
\end{rem}

\subsection{The Aronson-Serrin interpolated condition (iii)} 
\begin{thm}
Let $u_1,u_2$ be two solutions of \eqref{eq:HJ} with \(D_pH\in L^R(0,T;L^Q(\T^n))\) such as \(\frac{n}{2Q}+\frac{1}{P}\leq\frac12\) , terminal data $g_1,g_2\in L^p(\T^n)$ and $f_1,f_2\in L^1(0,T;L^p(\T^n))$. Then for all $t\in(0,T)$ 
\[
\|(u_1-u_2)(t)\|_{L^p(\T^n)}\leq C\left(\|g_1-g_2\|_{L^p(\T^n)}+\int_0^T\|f_1(s)-f_2(s)\|_{L^p(\T^n)}\,ds\right).
\]
\end{thm}
\begin{proof}
We follow step by steps  the proof of Theorem \ref{th:ii} until the equation \eqref{eq:bxt}. Finally by Remark \ref{rm:AronsonSerrin} we get the result.
\end{proof}


\begin{thebibliography}{10}

\bibitem{ODEbook}
R.~P. Agarwal and V.~Lakshmikantham.
\newblock {\em Uniqueness and nonuniqueness criteria for ordinary differential
  equations}, volume~6 of {\em Ser. Real Anal.}
\newblock Singapore: World Scientific, 1993.

\bibitem{Ambrosio04}
L.~Ambrosio.
\newblock Transport equation and {C}auchy problem for {$BV$} vector fields.
\newblock {\em Invent. Math.}, 158(2):227--260, 2004.

\bibitem{ACFS}
L.~Ambrosio, G.~Crippa, A.~Figalli, and L.~V. Spinolo.
\newblock Some new well-posedness results for continuity and transport
  equations, and applications to the chromatography system.
\newblock {\em SIAM J. Math. Anal.}, 41(5):1890--1920, 2009.

\bibitem{AronsonSerrin}
D.~G. Aronson and J.~Serrin.
\newblock Local behavior of solutions of quasilinear parabolic equations.
\newblock {\em Arch. Rational Mech. Anal.}, 25:81--122, 1967.

\bibitem{AronsonVazquez}
D.~G. Aronson and J.~L. V\'azquez.
\newblock The porous medium equation as a finite-speed approximation to a
  {H}amilton-{J}acobi equation.
\newblock {\em Ann. Inst. H. Poincar\'e{} Anal. Non Lin\'eaire}, 4(3):203--230,
  1987.

\bibitem{Barles}
G.~Barles, A.~Quaas, and A.~Rodr{\'{\i}}guez-Paredes.
\newblock Large-time behavior of unbounded solutions of viscous
  {Hamilton}-{Jacobi} equations in {{\(\mathbb{R}^N\)}}.
\newblock {\em Commun. Partial Differ. Equations}, 46(3):547--572, 2021.

\bibitem{Benton}
S.~H.~j. Benton.
\newblock {\em The {Hamilton}-{Jacobi} equation. {A} global approach}, volume
  131 of {\em Math. Sci. Eng.}
\newblock Elsevier, Amsterdam, 1977.

\bibitem{Bianchini2017}
S.~Bianchini, M.~Colombo, G.~Crippa, and L.~V. Spinolo.
\newblock Optimality of integrability estimates for advection-diffusion
  equations.
\newblock {\em NoDEA, Nonlinear Differ. Equ. Appl.}, 24(4):19, 2017.
\newblock Id/No 33.

\bibitem{BOP}
L.~Boccardo, L.~Orsina, and A.~Porretta.
\newblock Some noncoercive parabolic equations with lower order terms in
  divergence form.
\newblock {\em J. Evol. Equ.}, 3(3):407--418, 2003.

\bibitem{BKRS}
V.~I. Bogachev, N.~V. Krylov, M.~R{\"o}ckner, and S.~V. Shaposhnikov.
\newblock {\em Fokker-Planck-Kolmogorov equations}, volume 207 of {\em Math.
  Surv. Monogr.}
\newblock Providence, RI: American Mathematical Society (AMS), 2015.

\bibitem{BonicattoCiampaCrippa2022}
P.~Bonicatto, G.~Ciampa, and G.~Crippa.
\newblock On the advection-diffusion equation with rough coefficients: weak
  solutions and vanishing viscosity.
\newblock {\em J. Math. Pures Appl. (9)}, 167:204--224, 2022.

\bibitem{BonicattoCiampaCrippa2024}
P.~Bonicatto, G.~Ciampa, and G.~Crippa.
\newblock Weak and parabolic solutions of advection-diffusion equations with
  rough velocity field.
\newblock {\em J. Evol. Equ.}, 24(1):16, 2024.
\newblock Id/No 1.

\bibitem{CaffarelliSouganidis}
L.~A. Caffarelli and P.~E. Souganidis.
\newblock Rates of convergence for the homogenization of fully nonlinear
  uniformly elliptic {PDE} in random media.
\newblock {\em Invent. Math.}, 180(2):301--360, 2010.

\bibitem{CirantGoffi2020}
M.~Cirant and A.~Goffi.
\newblock Lipschitz regularity for viscous {H}amilton-{J}acobi equations with
  {$L^p$} terms.
\newblock {\em Ann. Inst. H. Poincar\'{e} C Anal. Non Lin\'{e}aire},
  37(4):757--784, 2020.

\bibitem{CirantGoffiAPDE}
M.~Cirant and A.~Goffi.
\newblock Maximal {$L^q$}-regularity for parabolic {H}amilton-{J}acobi
  equations and applications to mean field games.
\newblock {\em Ann. PDE}, 7(2):Paper No. 19, 40, 2021.

\bibitem{Cockburnetal}
B.~Cockburn, G.~Gripenberg, and S.-O. Londen.
\newblock Continuous dependence on the nonlinearity of viscosity solutions of
  parabolic equations.
\newblock {\em J. Differential Equations}, 170(1):180--187, 2001.

\bibitem{Constantin10}
A.~Constantin.
\newblock On {N}agumo's theorem.
\newblock {\em Proc. Japan Acad. Ser. A Math. Sci.}, 86(2):41--44, 2010.

\bibitem{CrandallIshiiLions}
M.~G. Crandall, H.~Ishii, and P.-L. Lions.
\newblock User's guide to viscosity solutions of second order partial
  differential equations.
\newblock {\em Bull. Amer. Math. Soc. (N.S.)}, 27(1):1--67, 1992.

\bibitem{CrippaThesis}
G.~Crippa.
\newblock {\em The flow associated to weakly differentiable vector fields},
  volume~12 of {\em Tesi. Scuola Normale Superiore di Pisa (Nuova Series)
  [Theses of Scuola Normale Superiore di Pisa (New Series)]}.
\newblock Edizioni della Normale, Pisa, 2009.

\bibitem{DDaners2000}
D.~Daners.
\newblock Heat kernel estimates for operators with boundary conditions.
\newblock {\em Math. Nachr.}, 217:13--41, 2000.

\bibitem{DipernaLions}
R.~J. DiPerna and P.-L. Lions.
\newblock Ordinary differential equations, transport theory and {S}obolev
  spaces.
\newblock {\em Invent. Math.}, 98(3):511--547, 1989.

\bibitem{Douglis}
A.~Douglis.
\newblock The continuous dependence of generalized solutions of non-linear
  partial differential equations upon initial data.
\newblock {\em Commun. Pure Appl. Math.}, 14:267--284, 1961.

\bibitem{Evans2007}
L.~C. Evans.
\newblock Lectures on kinetic formulations of nonlinear {PDE}.
\newblock In {\em Recent developments in nonlinear partial differential
  equations. Proceedings of the second symposium on analysis and PDEs, West
  Lafayette, IN, USA, June 7--10, 2004}, pages 1--24. Providence, RI: American
  Mathematical Society (AMS), 2007.

\bibitem{EvansARMA}
L.~C. Evans.
\newblock Adjoint and compensated compactness methods for {H}amilton-{J}acobi
  {PDE}.
\newblock {\em Arch. Ration. Mech. Anal.}, 197(3):1053--1088, 2010.

\bibitem{FlemingJDE}
W.~H. Fleming.
\newblock The {Cauchy} problem for a nonlinear first order partial differential
  equation.
\newblock {\em J. Differ. Equations}, 5:515--530, 1969.

\bibitem{GigaGiga}
M.-H. Giga and Y.~Giga.
\newblock {\em A basic guide to uniqueness problems for evolutionary
  differential equations}.
\newblock Compact Textb. Math. Cham: Birkh{\"a}user, 2023.

\bibitem{GTbook}
D.~Gilbarg and N.~S. Trudinger.
\newblock {\em Elliptic partial differential equations of second order}.
\newblock Classics in Mathematics. Springer-Verlag, Berlin, 2001.
\newblock Reprint of the 1998 edition.

\bibitem{GoffiTralli}
A.~Goffi and G.~Tralli.
\newblock Global geometric estimates for the heat equation via duality methods.
\newblock Preprint, {arXiv}:2409.15456, 2024.

\bibitem{LSU}
O.~A. Ladyzenskaja, V.~A. Solonnikov, and N.~N. Ural'tseva.
\newblock {\em Linear and quasilinear equations of parabolic type}.
\newblock Translated from the Russian by S. Smith. Translations of Mathematical
  Monographs, Vol. 23. American Mathematical Society, Providence, R.I., 1968.

\bibitem{LeBrisLions}
C.~Le~Bris and P.-L. Lions.
\newblock {\em Parabolic equations with irregular data and related issues.
  {Applications} to stochastic differential equations}, volume~4 of {\em De
  Gruyter Ser. Appl. Numer. Math.}
\newblock Berlin: De Gruyter, 2019.

\bibitem{LeonariPeralPrimoSoria}
T.~Leonori, I.~Peral, A.~Primo, and F.~Soria.
\newblock Basic estimates for solutions of a class of nonlocal elliptic and
  parabolic equations.
\newblock {\em Discrete Contin. Dyn. Syst.}, 35(12):6031--6068, 2015.

\bibitem{L82Book}
P.-L. Lions.
\newblock {\em Generalized solutions of {H}amilton-{J}acobi equations},
  volume~69 of {\em Research Notes in Mathematics}.
\newblock Pitman (Advanced Publishing Program), Boston, Mass.-London, 1982.

\bibitem{Lioduke}
P.-L. Lions.
\newblock Neumann type boundary conditions for {Hamilton}-{Jacobi} equations.
\newblock {\em Duke Math. J.}, 52:793--820, 1985.

\bibitem{LioAA}
P.-L. Lions.
\newblock Regularizing effects for first-order {H}amilton-{J}acobi equations.
\newblock {\em Applicable Anal.}, 20(3-4):283--307, 1985.

\bibitem{LionsSeminar}
P.-L. Lions.
\newblock Sur les \'equations de transport.
\newblock Recorded videolectures of a course at Coll\'ege de France, available
  at \url{https://www.college-de-france.fr/fr/enseignements/audios-videos},
  2012-2013, 2021-2022.

\bibitem{LionsSeegerARMA}
P.-L. Lions and B.~Seeger.
\newblock Transport equations and flows with one-sided {L}ipschitz velocity
  fields.
\newblock {\em Arch. Ration. Mech. Anal.}, 248(5):Paper No. 86, 61, 2024.

\bibitem{LionsSeegerAIHP}
P.-L. Lions and B.~Seeger.
\newblock Linear and nonlinear transport equations with coordinate-wise
  increasing velocity fields.
\newblock {\em Ann. Inst. H. Poincar\'e{} C Anal. Non Lin\'eaire},
  42(4):971--1036, 2025.

\bibitem{MetafuneJFA}
G.~Metafune, D.~Pallara, and A.~Rhandi.
\newblock Global properties of invariant measures.
\newblock {\em J. Funct. Anal.}, 223(2):396--424, 2005.

\bibitem{NazarovUraltseva}
A.~I. Nazarov and N.~N. Ural'tseva.
\newblock The {Harnack} inequality and related properties for solutions of
  elliptic and parabolic equations with divergence-free lower-order
  coefficients.
\newblock {\em St. Petersbg. Math. J.}, 23(1):93--115, 2012.

\bibitem{Ni59}
L.~Nirenberg.
\newblock On elliptic partial differential equations.
\newblock {\em Ann. Scuola Norm. Sup. Pisa Cl. Sci. (3)}, 13:115--162, 1959.

\bibitem{Otto}
F.~Otto.
\newblock {{\(L^ 1\)}}-contraction and uniqueness for quasilinear
  elliptic-parabolic equations.
\newblock {\em J. Differ. Equations}, 131(1):20--38, 1996.

\bibitem{Perron}
O.~Perron.
\newblock Eine hinreichende {Bedingung} f{\"u}r die {Unit{\"a}t} der
  {L{\"o}sung} von {Differentialgleichungen} erster {Ordnung}.
\newblock {\em Math. Z.}, 28:216--219, 1928.

\bibitem{PorrettaARMA}
A.~Porretta.
\newblock Weak solutions to {F}okker-{P}lanck equations and mean field games.
\newblock {\em Arch. Ration. Mech. Anal.}, 216(1):1--62, 2015.

\bibitem{PorrettaPriola}
A.~Porretta and E.~Priola.
\newblock Global {L}ipschitz regularizing effects for linear and nonlinear
  parabolic equations.
\newblock {\em J. Math. Pures Appl. (9)}, 100(5):633--686, 2013.

\bibitem{Silvestre}
L.~Silvestre.
\newblock On the differentiability of the solution to the {Hamilton}-{Jacobi}
  equation with critical fractional diffusion.
\newblock {\em Adv. Math.}, 226(2):2020--2039, 2011.

\bibitem{SouganidisJDE}
P.~E. Souganidis.
\newblock Existence of viscosity solutions of {Hamilton}-{Jacobi} equations.
\newblock {\em J. Differ. Equations}, 56:345--390, 1985.

\end{thebibliography}

\end{document}